\pdfoutput=1
\documentclass{amsart}
\usepackage{verbatim,amsfonts,color,mathrsfs,stmaryrd, bbold}
\usepackage{pdfpages}

\newtheorem{Thm}{Theorem}
\newtheorem{Cor}{Corollary}

\newtheorem{Lem}{Lemma}
\newtheorem{Prop}{Proposition}

\theoremstyle{definition}
\newtheorem{Def}{Definition}

\theoremstyle{remark}

\begin{document}

 
\title{Continued fractions for rational torsion}

\author{Katthaleeya Daowsud} 
\address{Department of Mathematics, Faculty of Science\\
Kasetsart University\\50 Ngam Wong Wan Rd, Ladyaow, Chatuchak\\
Bangkok, 10900, Thailand}
\email{fsciky@ku.ac.th}

\author{Thomas A. Schmidt}
\address{Oregon State University\\Corvallis, OR 97331}
\email{toms@math.orst.edu}
\keywords{continued fractions, genus 2 curves, torsion, Jacobians, rational point of order 11}
\subjclass[2010]{11Y65, 11G30, 14H40, 14Q05}
\date{18 August 2022}

\begin{abstract}   We exhibit a method to use continued fractions in function fields to find new families of hyperelliptic curves over the rationals with given torsion order in their Jacobians.    To show the utility of the method, we exhibit a new infinite family of curves over $\mathbb Q$ with genus two whose Jacobians have  torsion order eleven. 

 {\bf  In this updated version, we correct an error in the initial paper:}  The ``new" family claimed in the original Theorem~1  was pointed out by Professor D.~Lorenzini to have elements isomorphic with elements in Flynn's family (as defined in the paper); his guess that the families were the same up to element-wise isomorphism is correct.  Here, we give a family that is new; the old proof, now free of clerical error (we had replaced our $g_u(x)$ by $1+g_u(x)$ when determining Igusa invariants), holds.   Changes from the original version are flagged by {\color{red} UPDATED}; there are three of these: the new family $g_u(x)$; its partial quotients; the way to solve in the naive approach to find this new family.   (We also added an acknowledgement section.) Finally, as an appendix we include a PDF export of Maple calculations verifying our correction.
\end{abstract}

\maketitle


\section{Introduction}

As Cassels and Flynn express in the introduction to their text \cite{CasselsFlynn96}, there is still a need for interesting examples of curves of low genus over number fields.   Here, we show that the decidely  ``low brow'' method of continued fractions over function fields continues to have much to offer.     

We show that with a fixed base field, (low) genus and desired (small) torsion order,  one can search fairly easily for hyperelliptic curves of the given genus over the field whose divisor at infinity is of the given order.   As we recall with more details below,  the divisor at infinity has finite order in the Jacobian of the hyperelliptic curve if and only if  a corresponding continued fraction expansion,  in polynomials,  is periodic;  the order itself is  the sum of the degrees of initial partial quotients.   Given genus $g$ and torsion order $N$,  there are then finitely many possible partitions for the degrees of these initial partial quotients;  by making appropriate choices relating the coefficients of these partial quotients,  it is often possible to determine a curve with the desired genus and order. 

It has been known since 1940 \cite{BillingMahler40}  that 11 is the smallest prime for which there is no elliptic curve defined over the rationals with rational point of order the prime.  
Thus, our continued fraction approach must certainly fail with $k= \mathbb Q$, and $N=11$,  $g=1$.    It is natural to ask about $N=11$ and higher genus.   Indeed,  using a different method, Flynn \cite{Flynn90}, \cite{Flynn91} gave a one-dimensional family $\mathcal F_t$, see \eqref{e:flynn},    of hyperelliptic curves with $g=2$ and $N=11$.  Much more recently,  Bernard {\em et al} \cite{BernardLeprevostPohst09}  found some 18 additional individual curves with $(g,N) = (2, 11)$. (They state that they have found 19, but their table of results  lists one curve twice).    They explicitly state that they sought  infinite families of such curves.   

We exhibit a new infinite family  of this type. 
 {\color{red} UPDATED:} Let 
\begin{equation}\label{eq:NewFamily}
\begin{aligned}
x^6&-\dfrac{16 \,u}{u^5+8}\, x^5 -\dfrac{u^{15}+40 \,u^{10}+128 \,u^5+512}{2 \,u^3\, (u^5+8)^2} \,x^4 + \dfrac{ 6 \,u^{15}+ 224\, u^{10}+1536\, u^5+3072}{u^2 \,(u^5+8)^3}\,x^3\\
\\
&+ \dfrac{u^{30}+112 \,u^{25}+2880 \, u^{20}+25600\,  u^{15}+106496 \, u^{10}+262144 \,u^5+262144}{16 \,u^6 \,(u^5+8)^4} \, x^2 \\     
\\
&-\dfrac{u^{25}+80 \,u^{20}+1664 \,u^{15}+12288 \,u^{10}+36864 \,u^5+32768}{2 \, u^5 \, (u^5+8)^4}\, x\\
\\
& - \dfrac{u^{25}+46 \, u^{20}+736 \, u^{15}+5248\, u^{10}+ 18432 \, u^5+ 24576}{2 \, u^4 \,(u^5+8)^4}.
\end{aligned} 
\end{equation}

\begin{Thm}\label{t:newFam}    For each $u\in \mathbb Q\setminus\{0\}$, let  $\mathcal G_u$ be the smooth projective curve of affine equation $y^2 = g_u(x)$.  Then the divisor at infinity of the Jacobian of $\mathcal G_u$ has order 11.  There are infinitely many non-isomorphic $\mathcal G_u$, none of which is isomorphic to any of Flynn's curves $\mathcal F_t$.
\end{Thm} 

The proof that the torsion orders are 11 is given in  Lemma~\ref{l:newFamsCFexpansion}.  In Subsection~\ref{ss:NewInG2} we sketch the computation that this is a new infinite family.   Our naive method which led us to this new family of curves (and other similar curves) is discussed in Sections~\ref{s:improvedBound} through ~\ref{s:Applications}.

\bigskip
That finite torsion order is related to periodicity of continued fraction expansions   is a notion that can be traced back to Abel and Chebychev.  We first learned of this history, and the  relationship itself,  from work of Adams and Razar \cite{AdamsRazar80}.   Other authors who have discussed these notions include Berry \cite{Berry90} and van der Poorten with various coauthors, see {\em e.g.} \cite{vanderPoortenTran00}, \cite{vanderPoortenPappalardi05}.    See also the recent work of Platonov \cite{Platonov}.

The study of the arithmetic of function fields over finite fields goes back at least to E.~Artin's Ph.D. dissertation, \cite{Artin24}.   Much more recently,  
Friesen in particular has studied the structure of class groups using continued fractions, see say \cite{Friesen92}.    Our method can be viewed as a variant of that used by Friesen;  whereas he solves for the initial partial quotient in terms of the remaining terms of a given (quasi)-period, for small genus we find it more practical to solve for a quasi-period (or period) satisfying small sets of constraints. 
 
\section{Continued fractions and torsion at infinity} 

\subsection{Divisor at infinity}
If $k$ is a field of  characteristic zero (or sufficiently large),  then each hyperelliptic curve $\mathcal C$ of even genus $g$ over $k$ is $k$-isomorphic to a curve of affine equation 
\begin{equation}\label{eq:hyperellFormula}\mathcal C:  y^2 = u_{0} x^{2g+2} + u_{1} x^{2g+1} + \cdots + u_{2 g+ 1} x + u_{2 g+ 2}\,
\end{equation}
with coefficients $u_i \in k$, and $u_0$ a square in $k$, see \cite{Mestre1991}.     The affine curve $\mathcal C$ can be completed to a projective curve, for which one can take a smooth model (which we also denote by $\mathcal C$).
      
The {\em divisor group} of $\mathcal C$, $\text{Div}(\mathcal C)$, is the free abelian group generated by the $\bar k$-points of $\mathcal C$.   The divisors of degree zero, $\text{Div}^{0}(\mathcal C)$ is the kernel of the homomorphism from the  group $\text{Div}(\mathcal C)$ to the integers  defined by sending a weighted sum of points to the sum of these weights.    A {\em rational function} on $\mathcal C$ is a function to $\mathbb P^1(\bar k)$ that can be locally expressed as the quotient of polynomials,  the {\em divisor of a rational function} $\phi$ is the element of $\text{Div}(\mathcal C)$ given by the difference of the  zeros and poles of $\phi$, with multiplicity.      It is a classical result that the divisors of the rational functions on $\mathcal C$ form a subgroup of $\text{Div}^{0}(\mathcal C)$, called the subgroup of {\em principal divisors}.   The basic algebraic definition of the {\em Jacobian} of $\mathcal C$ is as the quotient group
\[ \text{Jac}(\mathcal C) = \text{Div}(\mathcal C)/\text{Div}^{0}(\mathcal C)\,.\]

The completion of the affine model of $\mathcal C$ leads to   two points at infinity, say $P, Q$.     Their formal difference then defines the {\em divisor at infinity},  $D_{\infty} = P-Q$,  which we take to be the corresponding element of $\text{Jac}(\mathcal C)$.   The  hyperelliptic involution, defined by  $(x,y) \mapsto (y,-x)$ interchanges $P$ and $Q$, and thus $D_{\infty}$ defines a point of $\text{Jac}(\mathcal C)$ defined over  $k$.    We say that $D_{\infty}$ is {\em   torsion of order $N$} if its class in 
$\text{Jac}(\mathcal C)$ has order $N$.

The goal of this work is to exhibit an elementary method to discover examples of hyperelliptic $\mathcal C$ of given genus $g$ whose  divisor at infinity is torsion of given order $N$.

\subsection{Continued fractions in function fields}  
 
   Given any field $k$,  the order of vanishing of a polynomial $f \in k[x]$ at the origin $x=0$ extends to define a valuation on the quotient field, the field of rational functions $k(x)$.       The valuation is simply given by writing any non-zero rational function as an integral power of $x$ times a rational function with neither zero nor pole at $x=0$; the exponent of $x$ is then the valuation of the initial rational function.    One defines a metric on $k(x)$ in the usual manner; the completion of this field with respect to the metric of the ring of formal Laurent series,   $k((x))$. 
   
   The point at infinity on the projective line over the field $k$ can be viewed as corresponding to the vanishing of $x^{-1}$.  A second valuation on $k(x)$,  leads to a completion that is $k((x^{-1}))$.    For  $\alpha \in k((x^{-1}))$, say 
\[ \alpha = c_{-n} x^n + c_{-n+1} x^{n-1} + \cdots  + c_{-1} x + c_0 + c_1 x^{-1} + c_2 x^{-2} + \cdots\,,\]
we define the {\em polynomial part} of $\alpha$ as 
\[ \lfloor\, \alpha\, \rfloor  = c_{-n} x^n + c_{-n+1} x^{n-1} + \cdots  + c_{-1} x + c_0\, .\]
We then define a continued fraction algorithm by way of the following sequences.  Let $\alpha_0 = \alpha$;   for $i \ge 0$ let $a_i =  \lfloor\, \alpha_i\, \rfloor$ and while $\alpha_i - a_i \neq 0$,  let $\alpha_{i+1} = ( \alpha_i - a_i)^{-1}$.       We then find an expansion of the form 
\begin{equation*}
\alpha=a_0+\cfrac{1}{a_1+\cfrac{1}{a_2+\cfrac{1}{a_3+\ddots}}} =:  [a_0; a_1, a_2, \dots, ]\,,
\end{equation*}
where we use the flat notation for typographic ease. The $a_i$ are called the {\em partial quotients} of the expansion.

The local ring of regular functions at a non-singular point of a projective curve is a discrete valuation ring.   In particular,  its maximal ideal, defined by vanishing at the point,  is principal.  Any generator of this maximal ideal is called a local uniformizer.    For    the point at infinity of $\mathbb P^1$ we can take $x^{-1}$ as a local uniformizer, thus leading to the valuation on $k(x)$  and the continued fractions as  above.  
 
   Our insistence on affine equations of the form of Equation ~\eqref{eq:hyperellFormula} is so that the degree two map from $\mathcal C$ to $\mathbb P^1$ defined by $(x,y) \mapsto x$ is such that the point at infinity of $\mathbb P^1$  has two pre-images.   The local rings at these points are then isomorphic to the local ring uniformized by $x^{-1}$.   From this,  it follows that there is a square root of $y^2 = f(x)$ in   $k((x^{-1}))$.   Indeed,  the completion of the ``rational field'' $k(x)$ with respect to the metric from our valuation is analogous to the completion of $\mathbb Q$ with respect to the usual metric.   Our $y = \sqrt{f(x)}$ is of the type that Artin \cite{Artin24}  called ``real quadratic'' --- exactly because it is a value in the completed field.   
  
Recall that the regular continued fraction expansion of $\sqrt{d}$ for non-square integers $d>0$ is periodic and with a palindromic period.   In general,  it is not true that $y = \sqrt{f(x)}$ has a periodic expansion,   but when it does the expansion shows similar symmetry.       In fact,  a new phenomenon arises: There may be further symmetry inside the period.

\begin{Def}   We call an even length sequence of $(a_1, \dots, a_{2 \ell})$ {\em skew symmetric} of {\em skew value} $\gamma$ if  
\[ a_{2(\ell -i)} = c_i\,a_{2i+1}\;\;\;\; \forall \, 0 \le i < \ell\,,\]
where 
\[ c_i = \begin{cases} \gamma &\mbox{if}\;\;2|i\,;\\
                                   \gamma^{-1}&\mbox{otherwise}\,,
            \end{cases}
\] 
with nonzero $\gamma \in k$.
\end{Def}

 Proofs of the following are given in \cite{Friesen92} when $k$ is a finite field, and in \cite{vanderPoortenTran00} in general; see also 
\cite{AdamsRazar80},  
\cite{Schmidt00} and \cite{vanderPoortenPappalardi05}.   We use an overline to denote a repetition of a sequence of partial quotients.  

\begin{Thm}\label{t:formPeriodic}   Suppose that a non-square $f(x) \in k[x]$ is of even degree, with leading coefficient a square in $k$.   If the continued fraction expansion of the Laurent series of $\sqrt{f(x)}$ is periodic,  then it is of the form 
\begin{equation}\label{eq:basicForm}
 \sqrt{f(x)}  = [\, a_0; \overline{a_1, \dots, a_{m-1}, 2 \kappa a_0 , a_{m-1}, \dots, a_1, 2 a_0}]\,,
 \end{equation}
where $(a_1, \dots, a_{m-1})$  is skew symmetric of skew value $\kappa$.   
\end{Thm} 

\begin{Def}    In the setting of Theorem~\ref{t:formPeriodic},  we call (the minimal) $m$ the {\em quasi-period} length.  
   
Exactly when $\kappa = 1$, we have equality of the (minimal) period length with the quasi-period length.    
Furthermore,  if these lengths are unequal, then the quasi-period length $m$ must be odd; we then speak of a {\em strict} quasi-period length.  

In what follows, we use $m$ to denote the quasi-period length of an expansion, including the case where this is the period length. 
\end{Def}

\bigskip
Direct calculation gives the partial quotients for the members of the family $\mathcal G_u$. {\color{red}   UPDATED.}
 
\begin{Lem}\label{l:newFamsCFexpansion}   For $y^2 = g_u(x)$ and $u \in \mathbb Q\setminus\{0\}$, the continued fraction expansion of  $y$ is periodic of  period length  $m= 8$.  The initial partial quotients of this expansion are 
\[\begin{aligned}
a_0(x) &= x^3-\dfrac{8 u}{u^5+8}\, x^2 - \dfrac{u^{15}+40 \,u^{10}+256\,u^5+512}{4 u^3 (u^5+8)^2}\, x + \dfrac{u^{10}+24 \,u^5+64}{u^2 (u^5+8)^2}\\
a_1(x) &=u\,x -\dfrac{4\, u^2}{u^5+8}\\
a_2(x) &= - \dfrac{u^5+8}{2\,u^3}\, x+ \dfrac{2}{u^2}\\
a_3(x) &= u\, x\\
a_4(x) &=  u \, x^2  -\dfrac{u^2\, (u^5+24)}{2\,(u^5+8)}\,  x - \dfrac{2\,u^{10}+16\,u^5+128}{u^2\,(u^5+8)^2},
\end{aligned}
\]
with $a_8(x) = 2 a_0(x)$, and the rest of the period is determined by the palindromic property.     
\end{Lem} 

\subsection{Periodic continued fractions identify torsion divisors}

The following is given by Adams and Razar \cite{AdamsRazar80} in the genus $g=1$ case.    The general case is proven in \cite{vanderPoortenTran00}, see also \cite{vanderPoortenPappalardi05}.    W.~Schmidt \cite{Schmidt00} presents a large portion of this result in his Lemma~7.

\begin{Thm}\label{t:quasiPeriodAndOrder} Suppose that $\mathcal C$ is given by $y^2 = f(x)$, with $f(x)$ as in the right hand side of \eqref{eq:hyperellFormula}.    If $\sqrt{f(x)}$ has periodic continued fraction expansion as in \eqref{eq:basicForm}, then  the divisor at infinity $D_{\infty}$ is torsion of order 
\begin{equation} \label{e:relateNandDegrees} N = g+1 + \sum_{i=1}^{m-1} \, (\deg a_i)\,.\end{equation}
Furthermore,  $1 \le \deg a_i   \le g$ for each $1\le i \le m-1$.
 \end{Thm} 

For ease of reference,  we express the obvious corollary. 
\begin{Cor}\label{c:boundNnaive}  The torsion order $N$ of the divisor at infinity satisfies
\[g+m\le N \le  m g+1\,.\]
 \end{Cor} 

In light of Lemma~\ref{l:newFamsCFexpansion}, the following is a direct application of Theorem~\ref{t:quasiPeriodAndOrder}. 
\begin{Lem}  Let $\mathcal G_u$ be as in Theorem~\ref{t:newFam}. 
Then the  divisor at infinity of the Jacobian of $\mathcal G_u$ is torsion of order 11.
\end{Lem} 

\section{Genus two}   We show the usefulness of our continued fraction based approach mainly in the genus two setting.   We use  the Igusa invariants for distinguishing isomorphism classes in genus two. 
 
\subsection{Igusa invariants}

 The affine representation of $\mathcal C$ given in \eqref{eq:hyperellFormula} is unique up to a fractional linear transformation of $x$,  $x \mapsto (a x + b)/(c x + d)$ and the associated $y \mapsto e y/(c x + d)^{g+1}$, where $a,b,c,d \in k$, $ad-bc \neq 0$,  $e \in k^{*}$.  
   The invariants of Igusa \cite{Igusa60}, nearly those defined by Clebsch in 1872 (see \cite{Mestre1991}), uniquely identify isomorphism classes of curves of genus $g= 2$ by taking distinct values on the orbits of these fractional linear transformations. 
   
        Fix $f(x) = u_{0} x^{6} + u_{1} x^{5} + \cdots + u_{5} x + u_{6}$, of roots $\alpha_i \in \bar{k}$,  and write $(i j)$ for  by $\alpha_i - \alpha_j$.   Then,  as in say \cite{LauterYang}, we write  Igusa's invariants as
\[ j_1(\mathcal C)  = A^5/D, \;\;\;    j_2(\mathcal C)  = A^3B/D, \;\;\;  j_3(\mathcal C)  = A^4C/D\,,\]
where in Igusa's shorthand notation, 
\[ 
\begin{aligned} 
A &= u_{0}^2 \sum_{15}\, (12)^2(34)^2(56)^2\\
B &= u_{0}^4 \sum_{10}\, (12)^2(23)^2(31)^2 (45)^2(56)^2(46)^2\\
C &= u_{0}^6 \sum_{60}\, (12)^2(23)^2(31)^2 (45)^2(56)^2(46)^2(14)^2(25)^2(36)^2\\
D &= u_{0}^{10} \prod_{i<j}\, (ij)^2\,.
\end{aligned}
\]
Thus,  $D$ is the discriminant of $f(x)$, hence  for $\mathcal C$ to be non-singular, we must have  $D \neq 0$; the summands  given in the definition of the functions $A, B, C$ have subscripts indicating the number of distinct summands to be taken.   Since each of $A, B, C, D$ is symmetric in the roots of $f$,  each can be given in terms of the elementary symmetric polynomials of these six roots, and therefore in terms of the coefficients of $f(x)$.   Thus,  the invariants themselves are rational functions of these coefficients.   Their expressions are sufficiently complicated that we follow tradition and do not give them here.

\subsection{Known examples of torsion order eleven}\label{ss:KnownInG2} 
Recall that Flynn \cite{Flynn90, Flynn91} gave a family of genus two curves defined over $\mathbb Q$ whose Jacobians each has a torsion point defined over $\mathbb Q$ of order 11.  The family is given in terms of $t \in \mathbb Q$ by  $\mathscr F_t: y^2 = f_t(x)$, where 
\begin{equation}\label{e:flynn} f_t(x) = x^6 + 2 x^5 + (2 t + 3) x^4 + 2 x^3 + (t^2 + 1) x^2 + 2 t (1 - t) x +  t^2\,,
\end{equation}
for rational $t$.   With obvious notation, one finds that 

\[
\begin{aligned} A_t &= - 8 (3 - 16 t + 56 t^2 + 4 t^3) \\
B_t &= 4 (9 - 120 t + 1045 t^2 - 1120 t^3 + 539 t^4 + 448 t^5 + 16 t^6)\\
C_t &= -8 (27 - 492 t + 4328 t^2 - 21984 t^3 + 71544 t^4 - 115456 t^5 + 
   60168 t^6 + 29984 t^7 + 2688 t^8 + 64 t^9)\\
 D_t &=  
-4096 t^7 (9 - 104 t + 432 t^2 + 16 t^3)\,.
\end{aligned} \]
In particular,  as \cite{BernardLeprevostPohst09} similarly deduced,  since (in fact, each of) the invariants $ j_1(\mathcal C_t),  j_2(\mathcal C_t),  j_3(\mathcal C_t)$ are non-constant, the Flynn family does indeed include infinitely many non-isomorphic curves.

As mentioned above, Bernard, Lepr\'evost and Pohst \cite{BernardLeprevostPohst09} gave some 18  other curves over $\mathbb Q$ with divisor of infinity of torsion order 11.  

\subsection{New family}\label{ss:NewInG2} The family of genus 2 curves $\mathcal G_u$ shares no isomorphism class with any of the curves of Flynn's family, $\mathscr F_t$.   This is verified by computation involving the Igusa invariants;   the only values of $(t,u)$ such that  $(\,j_1(t), j_2(t), j_3(t)\,) = (\, j_1(u), j_2(u), j_3(u)\,)$ (where we mildly abuse notation for the sake of legibility) are trivial in the sense that the invariants are then zero or infinity.    We now sketch this computation.

  Denoting a numerator of a quotient by $\text{Numer}$,  let  $\text{res}_{12}(t)$ denote  the polynomial in $t$ given by taking the resultant with respect to the variable $u$  of $\text{Numer}(\,j_1(t)-j_1(u)\,)$ with $\text{Numer}(\,j_2(t)-j_2(u)\,)$.   Of course,  the only values of $(t,u)$ such that $(\,j_1(t), j_2(t))\,) = (\, j_1(u), j_2(u)\,)$ are with $t$ a zero of this resultant polynomial.     Similarly define $\text{res}_{13}(t)$.   These two resultant polynomials both divisible by certain powers of $t$,  $9 - 104 t + 432 t^2 + 
   16 t^3$ and $3 - 16 t + 56 t^2 + 4 t^3$.    The first two of these divides $D_t$, the third $A_t$, see the display under  \eqref{e:flynn},  and therefore the vanishing of any of them corresponds to a degenerate case.     

    The polynomial GCD of $\text{res}_{12}(t)$ and $\text{res}_{13}(t)$,  after dividing by the aforementioned powers of the  innocuous factors, is a nonzero constant.      Therefore,  the resultants have no common nontrivial zero and   there  is indeed no  non-trivial occurrence of $(\,j_1(t),\, j_2(t), \,j_3(t)\,) = (\, j_1(u),\, j_2(u), \, j_3(u)\,)$.
    
     Involved in the above is the fact that the Igusa invariants for the family $\mathcal G_u$ are non-constant, and thus $\mathcal G_u$ is indeed an infinite family of non-isomorphic genus two curves.

\section{Bounding torsion order and partial quotient degree}\label{s:improvedBound}
   The naive upper bound on the order of the divisor at infinity is almost never achieved.    

\begin{Thm}\label{t:newBound}    The torsion order $N$ of the divisor at infinity of $\mathcal C$ satisfies  
\[
N < 1 + m g\]
whenever   $g >1$  and $m>2$. 
 \end{Thm} 
In fact,  from Theorem~\ref{t:consecDegs}  below, it follows that $N< (\frac{5m}{6}+1) g + m/2$.   
Furthermore,  our results give clear restrictions on the ``search space" when we set out to determine the curves of a fixed genus with divisors at infinity of fixed order.  

\subsection{Restrictions on partial quotients: statements}

Further restrictions then simply an upper bound of degree $g$ are imposed on the degrees of the partial quotients  when one  fixes genus $g$, that is when $\deg(a_0) = g+1$.     In this subsection we list several results, with proofs delayed until the following subsection.
\begin{Thm}\label{t:consecDegs}    Let $m$ be the quasi-period length of the continued fraction of $y$ for the hyperelliptic curve $\mathcal{C}$ over a field $k$.    Then for any $1 \le j \le m$, 
\begin{equation}\label{eq:topValue}  \deg(a_{j-1}) + \deg( a_j)\le  g+1 + \deg(a_1)\,.\end{equation}
Equality in \eqref{eq:topValue}  holds when $j=m$, and never holds for any three consecutive values of $j$ with $j \le m/2$.  Furthermore, 
whenever equality fails, then 
\begin{equation}  \deg(a_{j-1}) + \deg( a_j)\le  g+1\,.\end{equation}
 \end{Thm} 

\bigskip
In fact, we expect that equality in \eqref{eq:topValue} never holds for 
two consecutive values of $j$ with $j\le m/2$.   However, this entails a long and tedious calculation, whereas we are able to prove the above statement in a fairly straightforward manner.  Furthermore,  this already provides a helpful improvement in the upper bound of the order of the divisor at infinity.      The proof is given in terms of a certain sequence of rational functions, $h_j$, see below.   (That the $h_j$ are in fact polynomials is show in Daowsud's Ph.D. thesis \cite{Daowsud}.)    Focusing on leading coefficients of the $a_i$,  we were lead to the $h_j$.    The first, $h_1$, comes from the following straightforward variant of Friesen's result.

The proof of the following lemma is purely algebraic; the result thus holds for example in the case of regular continued fraction expansions of square roots.    We use the notation $lc(p)$ to denote the leading coefficient of a polynomial $p(x)$.

\begin{Lem}\label{l:firstJ}  Suppose that $\sqrt{f(x)}$ has a continued fraction expansion as in \eqref{eq:basicForm}, with quasi-period length $m$, value $\kappa$, and period length $n$.    Let $p_j/q_j$ denote the approximants to the purely periodic 
$1/(\sqrt{f(x)}-a_0)$.   Then 
\[
  f(x) - a_{0}^{2} = q_{n-1}/p_{n-2} = q_{m-1}/( \kappa \,p_{m-2})\,.
\]   
In particular,  $  f(x) - a_{0}^{2} \in k[x]$ is of degree $g+1 - \deg(a_1)$ and has leading coefficient   $2/\text{lc}(a_1)$.                    
\end{Lem} 
We remark that  setting $p_{-2} = 0$, $p_{-1}= 1$,  $q_{-2} = 1$, $q_{-1} = 0$ gives
\begin{align*}
p_0&=a_1,\ \ &q_0&=1,\\
p_1&=a_2a_1+1, \ \ &q_1&=a_2;\\
\text{and for}\,j\ge 2:  p_j&=a_{j+1}p_{j-1}+p_{j-2}, \ \ &q_j&=a_{j+1}q_{j-1}+q_{j-2}.
\end{align*}\smallskip

 Recall that $c_j$ is related to the skew value by $c_j = \kappa^{\pm 1}$, depending upon the parity of $j$.
\begin{Def}\label{def:Hseq}  For $j \le m$, define the following rational functions:
\begin{equation}\label{e:defHsubJ}
\begin{aligned}
h_1&=\dfrac{c_1 \,q_{m-1}}{p_{m-2}}\,,\\
h_2&=\dfrac{c_2 p_{m-4} h_1 - q_{m-3}}{p_{m-3}}\,,\; \text{and}\\
h_j&=\dfrac{c_j (p_{j-3} h_{j-1} + q_{j-3}) p_{m-(j+2)} - p_{j-4}q_{m-(j+1)}}{p_{m-(j+1)} p_{j-4}}\,,
\end{aligned} 
\end{equation} 
for  $3\le j \le m$. 
\end{Def} 

The $h_j$  enjoy a palindromic symmetry, which in particular accounts for the hypothesis that $j\le m/2$ in our statements about consecutive vanishing values. Since we do not use this aspect here, we merely state that :  For $m-1\ge j\ge 2$,  $h_{m+1-j} = h_j$.

\bigskip
 
Note that the $h_j$ are certainly rational functions,  and thus each has a degree defined as usual as the difference between the degrees of numerator and denominator.

\begin{Prop}\label{p:formHj} For $1\le j \le m$,  $h_j$ has non-negative degree.  If $h_j$ is non-zero, then: 
\begin{align*} \deg{h_1} &= g+1 - \deg(a_1)\,;\\
\deg{h_j} &= g+1 - \deg(a_{j-1}) - \deg(a_j)\,\;\;\mbox{when} \; j>1\,.
\end{align*}
Furthermore,  for $j>1$ we have that $h_j = 0$ if and only if  $g+1 + \deg(a_1) = \deg(a_{j-1}) + \deg(a_j)$; in particular, $h_m = 0$.   Finally,  
$h_2 = 0$ if and only if $m\le 2$.
\end{Prop} 

\bigskip

\bigskip

   As an easy application, consider the ``greedy" situation (of partitioning $N-g-1$, see Step (2) of the Method of Section~\ref{s:NaiveMethod}) where we set $\delta_j = g$.  
We note that since $\delta_0+\delta_1 = 2g+1$ is never equal to $2g$ and  $\delta_0  \ge 2g$ can only be satisfied when $g=1$,   it follows that if $g>1$ and $m>2$ then the equality $\delta_1=g$ implies that any $\delta_{j-1}+\delta_j \le g+1$ and in particular we cannot have consecutive partial quotients of degree $g$.    This implies that 
the naive upper bound of Corollary~\ref{c:boundNnaive}  is almost never achieved, that is  Theorem~\ref{t:newBound} holds.
 
\subsection{Restrictions on partial quotients: proofs} 

\bigskip 
Theorem~\ref{t:newBound} directly follows from Theorem~\ref{t:consecDegs}.   The proof of Theorem~\ref{t:consecDegs} follows directly from considering Equation \eqref{e:relateNandDegrees}  with Proposition~\ref{p:formHj} and Lemma~\ref{l:noThreeInRow} (below).     The following three lemmas yield the result of Proposition~\ref{p:formHj}.

\begin{Def}
For ease of expression,  let us define for each $0\le j \le m$, 
\[ \delta_j = \deg(a_j)\,.\]
\end{Def}
Thus $\delta_0 = g+1$, $\delta_{m-j} = \delta_j$.  Also, we have 
$\deg(p_j) - \deg(q_j) =   \delta_1$ and  $\deg(p_j) - \deg(p_{j-1}) = \delta_{j+1}$ (and similarly for consecutively indexed $q_j$).

\bigskip

\begin{Lem}\label{l:hSeqDegrees} For $2\le j \le m$,  if $h_j$ is non-zero, then: 
\[\deg{h_j} = \delta_0- \delta_{j-1} - \delta_j \,.\]

Furthermore,  $h_j = 0$ if and only if  $\delta_0 + \delta_1 = \delta_{j-1} + \delta_j$; in particular, $h_m = 0$.   
 \end{Lem} 

\begin{proof}  We argue by induction; for brevity's sake, we do not write out the straightforward proofs for the base cases.     

Consider first the case of $h_j \neq 0$.  Since
  $\deg( p_{j-4}q_{m-(j+1)}) - \deg(p_{m-(j+1)} p_{j-4}\,)$ is clearly negative, the degree of the numerator  in $h_j$ as given in Equation~\eqref{e:defHsubJ}  is determined by its first summand. 
Since $\deg(p_{j-3} ) > \deg(q_{j-3})$, when $h_{j-1} \neq 0$   we have
\[
\begin{aligned}
 \deg(h_j) &= \deg( h_{j-1})  + \deg(p_{j-3}) + \deg(p_{m-(j+2)}) - \deg(p_{j-4}) - \deg(p_{m-(j+1)})\\
  &= \deg( h_{j-1})  +\big(  \deg(p_{j-3}) - \deg(p_{j-4})\big) + \big(\deg(p_{m-(j+2)})  - \deg(p_{m-(j+1)})\big)\\
  &=  \deg( h_{j-1})  + \delta_{j-2}   - \delta_{m-j} = \deg( h_{j-1})  + \delta_{j-2}   - \delta_{j}\\
  &= \delta_0 -\delta_{j-1} - \delta_j\,.
 \end{aligned} 
 \]

If $h_{j-1} = 0$,  then 
\[
\begin{aligned}
 \deg(h_j) &= \deg(q_{j-3}) + \deg(p_{m-(j+2)}) - \deg(p_{j-4}) - \deg(p_{m-(j+1)})\\
&= \big( \deg(q_{j-3}) - \deg(p_{j-4})\big) + \big(\deg(p_{m-(j+2)})  - \deg(p_{m-(j+1)})\big)\\
&= \big(  \delta_{j-2} -  \delta_1\big) -  \delta_{m-j} = \delta_{j-2} -  \delta_1 -  \delta_{j} \\
 &= \delta_0 -\delta_{j-1} - \delta_j\,.
 \end{aligned} 
 \]
This last equality is by the induction argument whose hypothesis is: $h_{j-1} = 0$ implies $\delta_{j-2} = \delta_0 + \delta_1-\delta_{j-1}$.

We turn now to the case  of $h_j = 0$.  We must have
\[
\begin{aligned}
0 &= \deg(\, (p_{j-3} h_{j-1} + q_{j-3}) p_{m-(j+2)} ) -  \deg( p_{j-4}q_{m-(j+1)})\\
&= \deg(\, (p_{j-3} h_{j-1} + q_{j-3}) - \deg(  p_{j-4})    + \deg(p_{m-(j+2)} ) -  \deg(q_{m-(j+1)})\\
&= \deg(\, p_{j-3} h_{j-1} + q_{j-3}) - \deg(  p_{j-4})   -  \delta_{m-j} +   \delta_1\,.
 \end{aligned} 
 \]
If also $h_{j-1} = 0$,  then we find that $\delta_{j-2} = \delta_j$;  again by induction, the result follows.   Otherwise, we have
\[0 =  \deg(\, p_{j-3} h_{j-1}) - \deg(  p_{j-4})   -  \delta_{m-j} +   \delta_1\,,\]
with $\deg(\,h_{j-1}\,) = \delta_0- \delta_{j-2} - \delta_{j-1}$.  The result thus also easily follows.
 \end{proof}
 
 \bigskip
\begin{Lem}\label{l:hTwoVanishes} We have $h_2 = 0$ if and only if $m\le2$.
 \end{Lem} 
\begin{proof} From the previous result, we need only show that $h_2 =0$ implies $m\le 2$.   The vanishing of $h_2$ gives that $g+1 + \delta_1 = \delta_1 + \delta_2$.  That is, $\deg \, a_2 = g+1$,  which  by Theorem~\ref{t:quasiPeriodAndOrder}  signals the end of the (quasi-)period. 
\end{proof}

 \bigskip 
 There is an alternate expression for $h_j$. This follows directly from Daowsud's thesis \cite{Daowsud},  see sections 3.1, 3.2 and especially (3.15).     
\begin{Prop}\label{p:secondExpressionHj}(Daowsud)  If $j \ge 5$, then 
\[ h_j = \dfrac{p_{j-3}(p_{j-4}h_{j-2}+ q_{j-4} -a_{j-1} p_{j-5} h_{j-1}) - p_{j-5}  q_{j-2}}{p_{j-5} p_{j-4}}\,.\]  
\end{Prop} 

\bigskip 
\begin{Lem}\label{l:noThreeInRow}  If $m>2$, then  for each $3\le j \le m/2$,    not  all three
$h_j, h_{j-1}, h_{j-2}$ are zero. 
 \end{Lem} 
\begin{proof}    We again skip the straightforward proofs in the cases of small $j$.   Suppose that $j\ge 6$ and  $h_j = h_{j-1} = 0$.  
From Proposition~\ref{p:secondExpressionHj}, we find 
\[ h_{j+1} = \dfrac{p_{j-2} q_{j-3} -  p_{j-4}q_{j-1}}{p_{j-3}p_{j-4}} \]
and 
 \[ h_{j-2} =   \dfrac{p_{j-5} q_{j-2} -  p_{j-3}q_{j-4}}{p_{j-3}p_{j-4}}.\] 
We thus consider
 \[ 
 \begin{aligned}
  p_{j-2} q_{j-3} -  p_{j-4}q_{j-1} - [p_{j-5} q_{j-2} -  p_{j-3}q_{j-4}] &=   p_{j-2} q_{j-3} + p_{j-3}q_{j-4} -  [p_{j-4}q_{j-1} +p_{j-5} q_{j-2}]\\
                                                                                                           &=   p_{j-2} q_{j-3} + p_{j-3}q_{j-4} -  [p_{j-4}q_{j-1} +p_{j-5} q_{j-2}]\\ 
                                                                                                           &= (a_{j-1}p_{j-3} + p_{j-4})q_{j-3} + p_{j-3}q_{j-4}\\
                                                                                                           &\;\;\;\; -  [p_{j-4}(a_j q_{j-2}+ q_{j-3}) +p_{j-5} q_{j-2}]\\ 
                                                                                                           &= p_{j-3} (a_{j-1}q_{j-3} + q_{j-4}) - (a_j p_{j-4} + p_{j-5}) q_{j-2}\\
                                                                                                           &= p_{j-3}q_{j-2} - (a_j p_{j-4} + p_{j-5}) q_{j-2}\,.
\end{aligned}
\]
We find that the above equals zero if and only if $a_j = a_{j-2}$.    Thus, $h_j = h_{j-1} = 0$ implies the equivalence of $h_{j-2} = h_{j+1}$  with $a_j = a_{j-2}$.

We now find that 
\[
\begin{aligned}\label{t.3}
   h_{j+1} p_{j-3}p_{j-4}&= p_{j-2} q_{j-3} -  p_{j-4}q_{j-1} \\
                                      &=  (a_{j-1}   p_{j-3}+p_{j-4} )  q_{j-3} -  p_{j-4}(a_j q_{j-2}    + q_{j-3})\\
                                      &= a_{j-1}   p_{j-3}q_{j-3} -  a_j q_{j-2}p_{j-4}\\
                                      &=  a_{j-1}  (a_{j-2}   p_{j-4}+p_{j-5} ) q_{j-3} -  a_j (a_{j-1} q_{j-3}    + q_{j-4})p_{j-4}\\
                                      &=  a_{j-1}  (a_{j-2} -a_j)  p_{j-4}  q_{j-3} + a_{j-1}  p_{j-5} q_{j-3} -  a_j   q_{j-4} p_{j-4}\,.
 \end{aligned} 
 \]
\noindent 
But, $h_j = h_{j-1} = 0$ implies that $\deg a_j = \deg a_{j-2}$ and in particular $\deg (a_{j-1} a_{j-2}) > \deg a_j$.   From this, if $a_{j-2} \neq a_j$  then $ a_{j-1}  (a_{j-2} -a_j)  p_{j-4}  q_{j-3} $ is 
the unique highest degree summand in the final line of this last display.    Thus, if $h_{j+1}=0$ then this term must vanish,  therefore $a_{j-2} = a_j$   and thus also $h_{j-2} = h_{j+1}$ must vanish.    
From this, we can continue with ever decreasing indices $k$ with $h_k = 0$,  until we reach a contradiction. 
 
\end{proof}

As we mentioned above, in her thesis \cite{Daowsud},  Daowsud also proves the following. 
\begin{Thm}  Each $h_j$ is a polynomial.                   
\end{Thm} 

\section{Solving for $f(x)$, a naive method}\label{s:NaiveMethod}   We deduce a naive method for finding curves over a field with low order torsion points on their Jacobians.    We assume that $g>1$ and $m>2$.

\medskip
\centerline{\bf Naive Method}

\smallskip
\noindent
{\tt Given:}  field $k$,  genus $g$,  torsion order $N$ (with $N > g+1$). 

\noindent
{\tt Searches for:}  $f(x)\in k[x]$ such that $y^2 = f(x)$ has divisor at infinity of order $N$. 

\bigskip 
\begin{enumerate}
\item Fix $m\in [N-g, (N-1)/g]$ 
\item Choose a symmetric partition $(\delta_1, \dots, \delta_{m-1})$ of $N-g-1$,  with each $g \ge \delta_i \ge 1$ 
and  $g+1 + \delta_1 = \delta_{j-1} + \delta_j$ satisfied at most twice in a row for indices less than $m/2$, and 
$g+1 \ge  \delta_{j-1} + \delta_j$ for all other such $j$.   Furthermore, if $m>2$, then  $g+1 \ge  \delta_{1} + \delta_2$ must hold. 
\item  Introduce variables $c_{i,j}$ for $1 \le i \le \lfloor m/2 \rfloor$ and $0\le j \le \delta_i$
\item Set $a_i = a_i(x) = \sum_{j=0}^{\delta_i}\, c_{i,j}\, x^j$,  $1 \le i \le \lfloor m/2 \rfloor$ 
\item If $m$ is even then set $\kappa = 1$, else assign a variable $\kappa$ taking values in $k\setminus\{0, 1\}$.
\item For $1 \le i \le \lfloor m/2 \rfloor$ set  $a_{m-i}(x) = \kappa^{\pm }a_i (x)$ with alternating powers of $\kappa$
\item Introduce variables $r_0, \dots, r_g$ and set $a_0 = a_0(x) = x^{g+1}+\sum_{j=0}^g\, r_j x^j$. 
\item Expand the various $p_j$ and $q_j$ in terms of the $a_i$. 
\item  Introduce  variables $b_0, \dots,  b_{\delta_0 - \delta_1 - 1}$
\item Set $h_1(x) = (2/c_{1,\delta_1}) \,x^{\delta_0 - \delta_1} + \sum_{i=0}^{\delta_0 - \delta_1-1}\, b_i x^i$.
\item Solve the equation $h_1 = \dfrac{\,q_{m-1}}{\kappa  p_{m-2}}$ in the form 
\[ \kappa  p_{m-2} h_1 - q_{m-3} - 2 \kappa a_0  q_{m-2} = 0\] 
for the various $b_j, r_j, c_{i,j}$ (and $\kappa$ as appropriate)  under the restriction that no $c_{i, \delta_i}$ vanishes.  
 
\end{enumerate}

By induction, one shows that collecting powers of $x$ in (11) leads to $N-\delta_1$ equations.   The number of variables is $M:= (\delta_0-\delta_1) + (\lfloor m/2\rfloor +  \sum_{i=1}^{\lfloor m/2 \rfloor} \, \delta_i ) + (m \pmod 2\,) + (g+1)$.    For an upper bound on $M$, we use $N \ge \lfloor m/2\rfloor +  \sum_{i=1}^{\lfloor m/2 \rfloor} \, \delta_i$ and $g \ge \delta_0-\delta_1$.   Similarly, for a (rough) lower bound, we use $\delta_0-\delta_1 \ge 1$, $\sum_{i=1}^{\lfloor m/2 \rfloor} \, \delta_i  \ge \frac{N-g-1}{2}\ge \frac{N-g-1}{2g}$,  and $\lfloor m/2\rfloor \ge \frac{N-1}{2g}\ge  \frac{N-g-1}{2g}$.   Therefore, the number $M$ of variables is such that
\[g+1 + N \ge M \ge g+1 + (N-1)/g\,.\]

\medskip

\section{Applications of the method}\label{s:Applications}   
\subsection{An impossible partition}    To illustrate this method,   we sketch how when $g=2, N=11, m= 6$ (this even $m$ is thus the period length), the partition of $N-g-1 =  2+1+2+1+2 = \delta_1 + \delta_2 + \delta_3  + \delta_4  + \delta_5$ gives no admissible solution (over any coefficient field).

For increased legibility, we set some variable names other than those given in the recipe. Let 
\begin{equation} 
\begin{aligned}
a_0&:=x^3+r_2x^2+r_1x+r_0\\
a_1&:=l_1x^2 +c_1x +k_1\\
a_2&:=l_2x+k_2 \\\
a_3&:=l_3x^2+c_3x+k_3\\
h_1&:=(2/l_1)  x+b_0\,  {\color{red} \text{ corrected}}, \text{and}\\
J&:=(p_4 h_1-q_3)-2a_0q_4\,,\\
\end{aligned}
\end{equation}
where  $l_i \neq 0$ for each $i$, and we must solve so that the polynomial
$J(x)$, which a priori is of degree 8,  in fact vanishes.     Eliminating coefficients of decreasing powers of $x$, leads to (in this order, from top to bottom) the admissible values
\begin{equation} 
\begin{aligned}
b_0&=\frac{-2(c_1 - r_2l_1)}{l_1^2} \\
k_1&=\frac{2(r_1 l_1^2-r_2 c_1l_1+c_1^2)}{l_1}\\
r_0&=\frac{- 2 c_1^2 l_1 r_2 -  c_1 l_2 l_1^2 r_1 + c_1 l_2 l_1^2 r_2^2 -  l_2 l_1^3 r_2 r_1 - l_1^2 + l_2c_1^3  }{l_1^3 l_2}\\
r_2&=\frac{(l_2 c_1 + l_1k_2)}{l_1l_2}
\end{aligned}
\end{equation}
 However, one finds that $J(x)$ then has its coefficient of $x^4$ being $-l_2 l_3^2$,  which cannot vanish!   Thus,  there is no genus two curve corresponding to this partition of $N-g-1$.

\subsection{New family in genus two}\label{ss:newFamily}  Of course, the naive method must in some cases lead to the existence of curves. {\color{red} UPDATED.}   Indeed, 
for $g=2, N=11, m= 8$   and the partition of   $N-g-1 =  1+1+1+2+1+1+1$,  we set $a_0, a_2$ as above, but now set $a_1:=l_1 x +k_1, a_3 := l_3x+k_3, a_4 := l_4x+c_4 x+ k_4$ and $J:=( p_6 h_1-q_5)-2  a_0 q_6$.  

Eliminating coefficients as above,  in order,
\[
\begin{aligned}
r_2 &= \frac{  b_1\, l_{1}^{2}+2\,k_1}{2\, l_1}\\
r_1 &= \frac{b_0 \,l_{1}^{2} l_2+ b_1 k_1 l_1 l_2 +2 }{2\, l_1 l_2}\\    
r_0  &= \frac{b_0 k_1 l_1 l_{2}^{2} + b_1 l_1 l_2 - 2 \, k_2}{2\, l_1 l_{2}^{2}}\\
b_0 &= -\frac{-b_1 k_2 l_1 l_2 l_3 - l_{2}^{3} l_3  +2 \, k_{2}^{2}  l_{3} - 2\, l_2}{l_1 l_{2}^{2}l_3}\\
b_1 &= -\frac{-k_1 l_{2}^{3} l_{3}^{2}+k_2 l_1 l_{2}^{2} l_{3}^{2}-2\, k_2 l_1 l_3 -2\, k_3 l_1 l_2}{ l_{1}^{2} l_2 l_3}\\
k_2 &= -\frac{k_{1}^{2}  l_2 l_3 - k_1 k_3  l_1 l_2 + l_{1}^{2} - l_1 l_3}{- k_1   l_1 l_3+ k_3 l_{1}^{2}}\\
k_1 &= \frac{k_3 l_1 l_2 l_{3}^{2} + 2\, l_{1}^{2}}{l_2 l_{3}^{3}}\\
c_4 &= -\frac{l_1 l_{2}^{3} l_{3}^6 - l_{2}^{3} l_{3}^{7} + 2\, l_2 l_{3}^{6} - 8\, k_3 l_{1}^{2}l_2 l_{3}^{2} - 16\, l_{1}^{3}}{4\, l_{1}^2 l_2 l_{3}^{2}}\\
k_4 &=[(l_{2}^{4} l_{3}^{12} - 2\, k_3 l_{1}^{2} l_{2}^{4} l_{3}^{8})(l_1 - l_3) + 2\, l_1 l_{2}^{2} l_{3}^{11} - 4\, l_{1}^{2} l_{2}^{2}l_{3}^{6} (k_3 l_{3}^{2} + l_{1}^{2}l_{2})\\
&\phantom{moveOverMorePlease}+   8\, l_{1}^{3}l_{2}^{2}l_{3}^{4} (k_{3}^{2}l_{1}+ l_2 l_{3}^{3} )+ 32\,l_{1}^{5} (k_3  l_2 l_{3}^{2}+ l_{1})]/8\, l_{1}^{4}l_{2}^{2}l_{3}^{5}
\end{aligned}
\]
brings us to where a choice of $l_1 =  l_3, l_3 = -\frac{l_{3}^5 + 8}{2\, l_{3}^3}$ results in a solution to $J(x) = 0$.   Computation shows that the Igusa invariants are independent of choice of the value of $k_3$; we set $k_3 = 0$ and $l_3 = u$ to give us the one dimensional family $\mathcal G_u$.

\subsection{Examples in genus greater than two}  In her thesis \cite{Daowsud} Daowsud gave 
examples for all possible genera  $g$ of curves defined over $\mathbb Q$ with Jacobian of divisor at infinity having torsion order 11.    For brevity's sake, we simply report the basic form of the expansions that she found when $g>2$.

\bigskip 
 
\noindent%
\begin{minipage}{\textwidth}
  \centering
\begin{tabular}{|c|c|c|}
\hline
$g$&$m$&$(\delta_0, \delta_1, \cdots, \delta_{m-1})$\\
\hline\hline
3&6&(4,2,1,1,1,2)\\
\hline
4&4&(5,1,2,2,1)\\
\hline
5&4&(6,1,3,1)\\
\hline
6&3&(7,2,2)\\
\hline
7&4&(8,1,1,1)\\
\hline
8&2&(9,2)\\
\hline
9&2&(10,1)\\
\hline
10&1&(11)\\
\hline\hline
\end{tabular}

  \medskip 

Table 1.  Genera, period length and partition of $N=1$ for examples in \cite{Daowsud} for higher genus.  
\end{minipage}
 
\vspace{.5 cm}
\section{Acknowledgment}  We thank Professor Dino Lorenzini   for pointing out counterexamples to the original theorem and asking  about the full isomorphism with Flynn's family.   We also gratefully acknowledge his encouragement to find a corrected new  family.

\section{Appendix:  Transcript of Maple supporting correction}
\includepdf[pages=-]{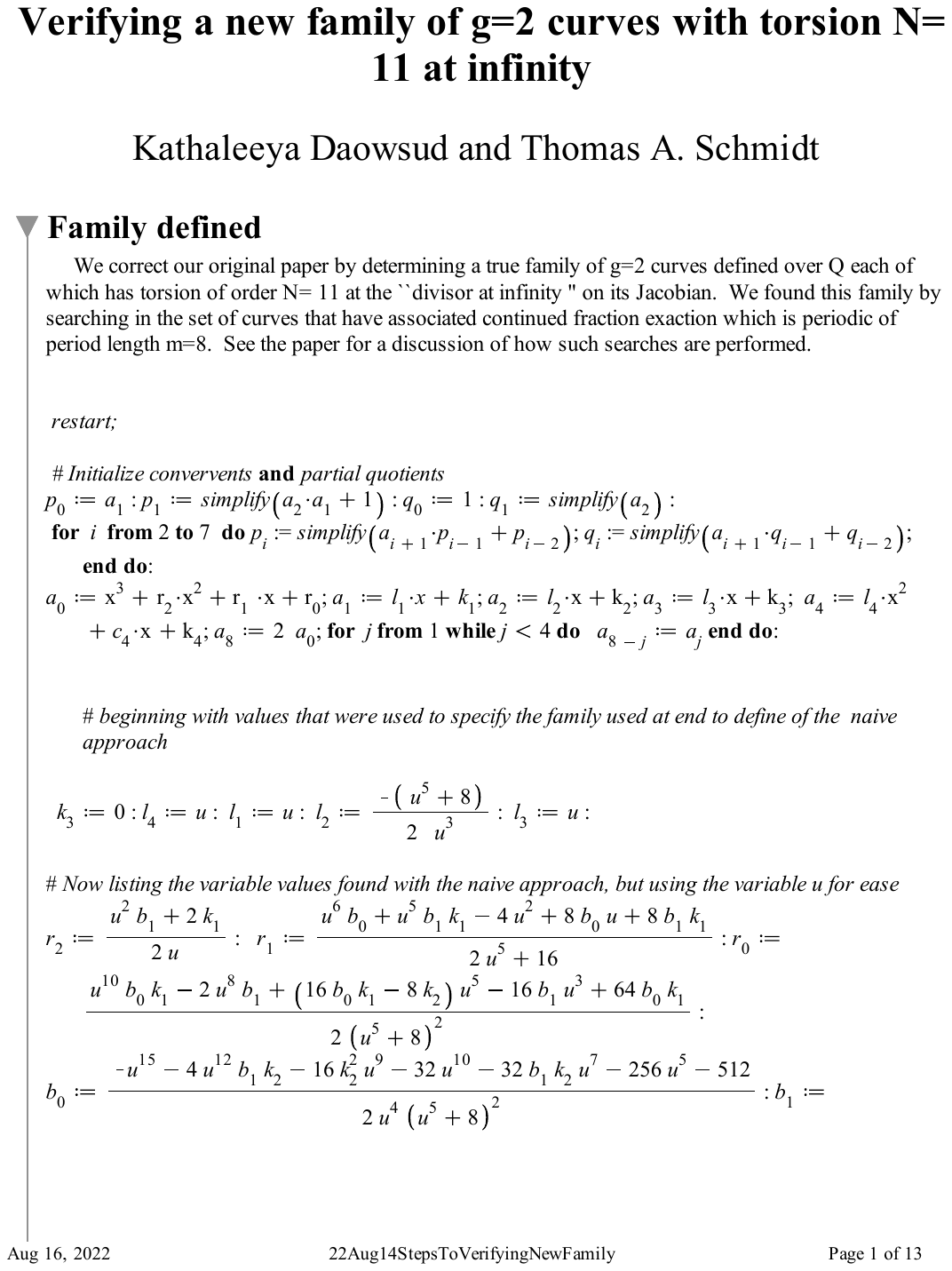}

\end{document}